\documentclass[]{amsart}

\usepackage{amssymb}
\usepackage{amsfonts}
\usepackage{amsmath}
\usepackage{amsthm}
\usepackage{graphicx}
\usepackage{mathrsfs}
\usepackage{dsfont}
\usepackage{amscd}
\usepackage{multirow}
\usepackage[all]{xy}
\normalfont
\usepackage[T1]{fontenc}
\usepackage{calligra}
\usepackage{verbatim}
\usepackage[usenames,dvipsnames]{color}
\usepackage[colorlinks=true,linkcolor=Blue,citecolor=Violet]{hyperref}
\usepackage{enumerate}

\newcommand{\N}{{\mathds{N}}}
\newcommand{\Z}{{\mathds{Z}}}
\newcommand{\Q}{{\mathds{Q}}}
\newcommand{\R}{{\mathds{R}}}
\newcommand{\C}{{\mathds{C}}}
\newcommand{\T}{{\mathds{T}}}

\newcommand{\D}{{\mathfrak{D}}}
\newcommand{\A}{{\mathfrak{A}}}
\newcommand{\B}{{\mathfrak{B}}}

\newcommand{\Lip}{{\mathsf{L}}}

\newcommand{\solenoid}[1]{\mathcal{S}_{\,#1}}

\newcommand{\dpropinquity}[1]{{\mathsf{\Lambda}^\ast_{#1}}}
\newcommand{\covpropinquity}[1]{{\mathsf{\Lambda}^{\mathrm{cov}}_{#1}}}

\newcommand{\Kantorovich}[1]{{\mathsf{mk}_{#1}}}

\newcommand{\KantorovichDist}[3]{{\mathrm{mkD}_{#1}\left({#2},{#3}\right)}}

\newcommand{\Haus}[1]{{\mathsf{Haus}_{#1}}}

\newcommand{\StateSpace}{{\mathscr{S}}}

\newcommand{\MongeKant}{{Mon\-ge-Kan\-to\-ro\-vich metric}}

\newcommand{\Lqcms}{{\JLL} quantum compact metric space}
\newcommand{\Qqcms}[1]{{$#1$}--\gQqcms}
\newcommand{\gQqcms}{quasi-Leibniz quantum compact metric space}

\newcommand{\qcms}{quantum compact metric space}
\newcommand{\lcqms}{quantum locally compact metric space}

\newcommand{\unit}{1}

\newcommand{\sa}[1]{{\mathfrak{sa}\left({#1}\right)}}

\newcommand{\UIso}[4]{{\mathsf{UIso}_{#1}\left({#2}\rightarrow{#3}\middle\vert{#4}\right)}}

\newcommand{\JLL}{Lei\-bniz}

\newcommand{\bridge}[1]{#1} 
\newcommand{\dom}[1]{{\operatorname*{dom}\left({#1}\right)}}
\newcommand{\codom}[1]{{\operatorname*{codom}\left({#1}\right)}}

\newcommand{\norm}[2]{\left\|{#1}\right\|_{#2}}
\newcommand{\tunnelset}[4]{{\text{\calligra Tunnels}\,\left[{#1}\stackrel{#3}{\longrightarrow}{#2}\middle\vert {#4} \right]}}
\newcommand{\tunnelsetc}[3]{{\text{\calligra Tunnels}\,\left[\left({#1}\right) \stackrel{#3}{\longrightarrow} \left({#2}\right) \right]}}

\newcommand{\bridgereach}[2]{{\varrho\left(\bridge{#1}\middle\vert {#2}\right)}}

\newcommand{\bridgeheight}[2]{{\varsigma\left(\bridge{#1}\middle\vert{#2}\right)}}

\newcommand{\bridgelength}[2]{{\lambda\left(\bridge{#1}\middle\vert{#2}\right)}}

\newcommand{\bridgenorm}[2]{{\mathsf{bn}_{ \bridge{#1}  }\left({#2}\right)}}

\newcommand{\worknote}[1]{} 
\newcommand{\opnorm}[3]{{\left|\mkern-1.5mu\left|\mkern-1.5mu\left| {#1} \right|\mkern-1.5mu\right|\mkern-1.5mu\right|_{#3}^{#2}}}

\newcommand{\tunnelreach}[2]{{\rho\left({#1}\middle\vert{#2}\right)}}

\newcommand{\tunnelmagnitude}[2]{{\mu\left({#1}\middle\vert{#2}\right)}}

\newcommand{\tunnelextent}[1]{{\chi\left({#1}\right)}}

\newcommand{\alg}[1]{{\mathfrak{#1}}}

\newcommand{\dil}[1]{{\mathrm{dil}\left({#1}\right)}}

\theoremstyle{plain}
\newtheorem{theorem}{Theorem}[section]

\newtheorem{corollary}[theorem]{Corollary}

\newtheorem{theorem-definition}[theorem]{Theorem-Definition}

\theoremstyle{definition}
\newtheorem{definition}[theorem]{Definition}

\newtheorem{example}[theorem]{Example}

\theoremstyle{remark}

\newtheorem{notation}[theorem]{Notation}

\renewcommand{\geq}{\geqslant}
\renewcommand{\leq}{\leqslant}

\numberwithin{equation}{section}

\allowdisplaybreaks[4]

\begin{document}

\title[Symmetries and the Dual Propinquity]{A survey of the Preservation of Symmetries by the Dual Gromov-Hausdorff Propinquity}
\author{Fr\'{e}d\'{e}ric Latr\'{e}moli\`{e}re}
\email{frederic@math.du.edu}
\urladdr{http://www.math.du.edu/\symbol{126}frederic}
\address{Department of Mathematics \\ University of Denver \\ Denver CO 80208}

\date{\today}

\begin{abstract}
We survey the symmetry preserving properties for the dual propinquity, under natural non-degeneracy and equicontinuity conditions. These properties are best formulated using the notion of the covariant propinquity when the symmetries are encoded via the actions of proper monoids and groups. We explore the issue of convergence of Cauchy sequences for the covariant propinquity, which captures, via a compactness result, the fact that proper monoid actions can pass to the limit for the dual propinquity.
\end{abstract}
\maketitle



The dual propinquity, a noncommutative analogue of the Gromov-Hausdorff distance, enjoys some intrinsic symmetry-preserving properties, which is particularly valuable for the study of many types of noncommutative geometries arising from group or semigroup actions, as well as the study of physical models where symmetries play a central role. In order to capture these elusive properties, a notion of covariant propinquity can be defined, which enables a discussion of convergence of actions of proper monoids, under the mild assumption that the actions are by Lipschitz maps for an underlying quantum metric. The covariant propinquity is an interesting metric in its own right, and it contributes to our overall research program, where we seek to encapsulate within certain hypertopologies on classes of quantum structures, various basic geometric components of potential physical interest. This note surveys the covariance property of the dual propinquity observed in \cite{Latremoliere17c} and the covariant propinquity as introduced in \cite{Latremoliere18b,Latremoliere18c}.

Metric considerations in noncommutative metric geometry can be traced back to Connes' introduction of spectral triples in \cite{Connes89}, and its link with mathematical physics has been a matter of interest for a long time. Parallel to these observations, the metric theory of groups \cite{Gromov81}, and applications of metric geometric ideas to Riemannian manifolds centered around the Gromov-Hausdorff distance \cite{burago01} have proven to be very powerful tools. Moreover, the Gromov-Hausdorff distance for compact metric spaces was in fact first introduced by Edwards \cite{Edwards75}, while considering the problem of defining a geometry for Wheeler's superspace as an approach to quantum gravity \cite{Wheeler68}. In other words, the use of metric ideas, and specifically, of topologizing appropriate classes of metric spaces using some version of the Gromov-Hausdorff distance, is a theme well worth exporting to noncommutative geometry.

This endeavor was initiated by Rieffel \cite{Rieffel98a,Rieffel99,Rieffel00}, who first introduced the appropriate notion of a {\qcms}, inspired by Connes' original work, and best understood as generalizing to noncommutative geometry the metric introduced by Kantorovich \cite{Kantorovich40,Kantorovich58} in his study of Monge's transportation problem. Rieffel then defined the quantum Gromov-Hausdorff distance \cite{Rieffel00} on his class of {\qcms s}. Some very nontrivial convergences of noncommutative spaces were established, including for instance finite dimensional approximations of quantum tori \cite{Latremoliere05}.

However, defining a noncommutative analogue of the Gromov-Hausdorff distance is not a trivial matter. For instance, Rieffel's distance could be null between non-isomorphic C*-algebras, as long as these are endowed with isometric quantum metric in a natural sense. This issue grew into more than a curiosity when the study of the behavior of structures associated to {\qcms s}, such as modules over such spaces, became a focus of the research in noncommutative metric geometry. Indeed, it became apparent that Rieffel's distance does not capture the entire C*-algebraic structure: in fact, the distance is defined on {\qcms s} built on top of order-unit spaces rather than algebras, and thus intermediate estimates of Rieffel's distance will often involve leaving the category of C*-algebras. It should however be noted that Rieffel's metric is very flexible and convenient to use, and therein lies its power; it also has a clear intuitive meaning in term of distance between states, a notion at the core of C*-algebra theory and its application to physics.

Several  interesting alternative to Rieffel's distance were introduced, often motivated by this coincidence property issue. A notable early example is due to Kerr \cite{Kerr02}, who replaced states with unital, completely positive maps from operator systems to matrix algebras. Kerr then proved that distance zero between two {\qcms s} defined over C*-algebras indeed implied that the underlying C*-algebras are *-isomorphic, in addition to the quantum metrics being isometric in the sense of Rieffel. An interesting phenomenon emerged from Kerr's work, where completeness for his metric relied on a relaxed notion of the Leibniz inequality which plays no role (nor can be made sense of) in Rieffel's notion of {\qcms s}. Yet, introducing this quasi-Leibniz condition breaks the proof of the triangle inequality as note in \cite{Kerr02}.

This problem with the Leibniz, or even relaxed Leibniz, condition, again grew from a curious observation to a recurrent issue when working, for instance, with the convergence of modules \cite{Rieffel10c}. Eventually, it became clear that one wished for a  noncommutative analogue of the Gromov-Hausdorff distance which would, at the same time, be compatible with C*-algebras, be compatible with Leibniz or quasi-Leibniz quantum metrics, and keep the underlying interpretation of Rieffel's pioneering metric in term of states. Taken altogether, these requirements meant solving the issues that plagued the field for its decade of existence.

The resolution of this challenge came in the form of our dual Gromov-Hausdorff propinquity \cite{Latremoliere13,Latremoliere13b,Latremoliere14,Latremoliere15}. The dual propinquity addresses all the above needs by working exclusively with {\qcms s} defined over C*-algebras and endowed with quasi-Leibniz quantum metrics, while being a complete metric up to full quantum isometry---meaning, in particular, that distance zero does imply *-isomorphism of the underlying C*-algebras. Since its introduction, the dual propinquity seems to have become the metric of choice for this subject. A particular version of the construction of our metric, called the quantum propinquity \cite{Latremoliere13}, which was in fact our first step toward this resolution, has proven particularly valuable.

Equipped with this new metric, we address some of the problems we hoped to study from a noncommutative metric geometric perspective, such as approximation of modules \cite{Latremoliere16c,Latremoliere17a,Latremoliere18a}, or other structures of value when eyeing applications to mathematical physics. Such structures include symmetries and dynamics, encoded by actions of groups or semigroups on {\qcms s}. It would appear very valuable to understand the interplay between convergence of {\qcms s} and convergence of some of their symmetries, for instance. Note that in some sense, such a study requires a Gromov-Hausdorff metric which does indeed discriminate between *-automorphisms.

This note surveys our progress regarding the problem of preservation of symmetries of dynamics under convergence for the dual propinquity. We approach this problem by actually taking our construction of the dual propinquity further to include entire proper monoid actions by Lipschitz maps, which requires just a small amount of changes to the basic concepts---hiding the difficulties in the proofs of the basic properties. This observation alone suggests that the dual propinquity is close to remembering something of symmetries by itself. This is indeed vindicated by a sort of Arz{\'e}la-Ascoli theorem where equicontinuity conditions ensure that convergence for the dual propinquity can be strengthened to convergence of entire monoid or group actions. The equicontinuity condition is expressed using natural concepts associated with quantum metrics and in particular, Lipschitz maps.

In the first section of this survey, we introduce the category of {\qcms s} and their Lipschitz morphisms. When discussing monoid actions, we certainly want to use the right notion of morphism, and it turns out that a natural picture emerges from the notion of {\qcms s}. We introduce the very natural ingredient which will then be used to express our equicontinuity conditions on actions. The second section describes the dual propinquity, as the basic object of our research. We then turn to the covariant propinquity in the last section. We construct this metric and then show a compactness-type result which formalizes the notion of preservation of symmetries by the dual propinquity. We conclude with a discussion of the completeness of the covariant propinquity and an example of covariant metric convergence.

\section{The category of {\qcms s}}

A quantum metric space is a noncommutative analogue of the algebra of Lipschitz functions over a metric space. The definition of such an analogue for compact metric spaces has evolved from an observation of Connes \cite{Connes89} to the corner stone of our approach to noncommutative metric geometry in a few key steps. Most importantly, Rieffel observed \cite{Rieffel98a,Rieffel99} that the metric on state spaces in \cite{Connes89} was in essence a special case of a noncommutative generalization of the construction of Kantorovich of a distance on Radon probabilities measures by duality from the Lipschitz seminorm \cite{Kantorovich40,Kantorovich58}.  The core features of the Kantorovich metric include the fact that it metrizes the weak* topology, and this is taken as a starting point for identifying those seminorms on C*-algebras which are candidates for noncommutative Lipschitz seminorms.

Another feature of Lipschitz seminorms is that they satisfy the Leibniz inequality. The importance of this property was not fully evident at the beginning of the study of {\qcms s}, though the difficulties it introduces in the study of noncommutative analogues of the Gromov-Hausdorff distance were \cite{Rieffel10c}. In fact, Rieffel's distance \cite{Rieffel00} does not require that noncommutative Lipschitz seminorms, then named Lip-norms, possess any such Leibniz properties. We only realized much later that some form of connection between the Lip-norms and the underlying multiplicative structure is in fact a means to define a noncommutative Gromov-Hausdorff distance which is zero only when the underlying C*-algebras are *-isomorphic. More importantly, as seen in our work, such a connection enables us to push forward our program by making it possible to study group actions on {\qcms s} and appropriately defined modules over {\qcms s}.

To express the connection between multiplication and noncommutative Lipschitz seminorms, we introduce, inspired by \cite{Kerr02}:

\begin{definition}
  A function $F : [0,\infty)^4 \rightarrow [0,\infty)$ is \emph{permissible} when $F$ is weakly increasing for the product order on $[0,\infty)^4$ and for all $x,y,l_x,l_y\geq 0$, we have $F(x,y,l_x,l_y) \geq x l_y + y l_x$.
\end{definition}

Our definition for {\qcms s} is thus as follows.

\begin{notation}
  Throughout this paper, for any unital C*-algebra $\A$, the norm of $\A$ is denoted by $\norm{\cdot}{\A}$, the space of self-adjoint elements in $\A$ is denoted by $\sa{\A}$, the unit of $\A$ is denoted by $\unit_\A$ and the state space of $\A$ is denoted by $\StateSpace(\A)$. We also adopt the convention that if a seminorm $\Lip$ is defined on some dense subspace of $\sa{\A}$ and $a\in\sa{\A}$ is not in the domain of $\Lip$, then $\Lip(a) = \infty$.
\end{notation}

\begin{definition}[{\cite{Connes89,Rieffel98a,Rieffel99,Rieffel10c,Latremoliere13,Latremoliere15}}]\label{qcms-def}
  An \emph{$F$-\qcms} $(\A,\Lip)$, for a permissible function $F$, is an ordered pair consisting of a unital C*-algebra $\A$ and a seminorm $\Lip$, called an \emph{L-seminorm}, defined on a dense Jordan-Lie subalgebra $\dom{\Lip}$ of $\sa{\A}$, such that:
  \begin{enumerate}
    \item $\{ a \in \sa{\A} : \Lip(a) = 0 \} = \R\unit_\A$,
    \item the \emph{\MongeKant} $\Kantorovich{\Lip}$ defined for any two states $\varphi, \psi \in \StateSpace(\A)$ by:
      \begin{equation*}
        \Kantorovich{\Lip}(\varphi, \psi) = \sup\left\{ |\varphi(a) - \psi(a)| : a\in \dom{\Lip}, \Lip(a) \leq 1 \right\}
      \end{equation*}
      metrizes the weak* topology restricted to $\StateSpace(\A)$,
    \item $\Lip$ satisfies the $F$-quasi-Leibniz inequality, i.e. for all $a,b \in \dom{\Lip}$:
      \begin{equation*}
        \max\left\{ \Lip\left(\frac{a b + b a}{2}\right), \Lip\left(\frac{a b - b a}{2i}\right) \right\} \leq F(\norm{a}{\A},\norm{b}{\A},\Lip(a),\Lip(b))\text{,}
      \end{equation*}
    \item $\Lip$ is lower semi-continuous with respect to $\norm{\cdot}{\A}$.
  \end{enumerate}
We say that $(\A,\Lip)$ is \emph{Leibniz} when $F$ can be chosen to be $F:x,y,l_x,l_y \mapsto x l_y + y l_x$.
\end{definition}

Definition (\ref{qcms-def}) is modeled after the following classical picture.

\begin{example}
  Let $(X,d)$ be a compact metric space. For any $f \in C(X)$, where $C(X)$ is the C*-algebra of continuous $\C$-valued functions over $X$, we set
  \begin{equation*}
    \Lip(f) = \sup\left\{\frac{|f(x)-f(y)|}{d(x,y)} : x,y \in X, x\not= y \right\} \text{,}
  \end{equation*}
  allowing for the value $\infty$. Of course, $\Lip$ is the usual Lipschitz seminorm on $C(X)$. Now, by \cite{Kantorovich40,Kantorovich58}, the {\MongeKant} $\Kantorovich{\Lip}$ does metrize the weak* topology on the state space $\StateSpace(C(X))$, which by Radon-Riesz theorem is the space of all Radon probability measures. Moreover, $\Lip(f g) \leq \norm{f}{C(X)}\Lip(g) + \Lip(f)\norm{g}{C(X)}$ for all $f,g \in C(X)$, so $(C(X),\Lip)$ is a Leibniz {\qcms}.
\end{example}

A truly noncommutative example of a Leibniz {\qcms} was obtained by Rieffel in \cite{Rieffel98a}.

\begin{example}[{\cite{Rieffel98a}}]\label{group-lip-ex}
  Let $G$ be a compact group with identity element $e$ and $\ell$ be a continuous length function over $G$. Let $\A$ be a unital C*-algebra such that there exists a strongly continuous action $\alpha$ of $G$ on $\A$. For all $a\in \A$, we define:
  \begin{equation*}
    \Lip(a) = \sup\left\{ \frac{\norm{\alpha^g(a) - a}{\A}}{\ell(g)} : g \in G\setminus\{e\}  \right\} \text{,}
  \end{equation*}
allowing for the value $\infty$. The condition that $\alpha$ is ergodic, i.e. $\{a\in\A:\forall g \in G\quad \alpha^g(a) = a\}=\C\unit_\A$, which is clearly necessary for $(\A,\Lip)$ to be a {\qcms}, is proven by Rieffel in \cite{Rieffel98a} to be sufficient as well.

This family of examples include:
\begin{itemize}
  \item \emph{quantum tori}, where $\A = C^\ast(\Z^d,\sigma)$ for some multiplier $\sigma$ of $\Z^d$ (with $d\geq 2$) and with $G = \T^d$ the $d$-torus, using the dual action $\alpha$,
  \item \emph{fuzzy tori}, where $G$ is a finite subgroup of $\T^d$, and $\A = C^\ast(\widehat{G},\sigma)$ for some multiplier $\sigma$ of $G$, again using the dual action as $\alpha$,
  \item \emph{noncommutative solenoids} \cite{Latremoliere16}, where $G = \solenoid{p}^2$ with $p \in \N\setminus\{0,1\}$:
    \begin{equation*}
      \solenoid{p} = \left\{ (z_n)_{n\in\N} \in \T^\N : \forall n \in \N \quad z_{n+1}^p = z_n \right\}
    \end{equation*}
    is the solenoid group, and where $\A = C^\ast\left(\Z\left[\frac{1}{p}\right]\times \Z\left[\frac{1}{p}\right], \sigma\right)$ for some multiplier $\sigma$ of $\Z\left[\frac{1}{p}\right]\times \Z\left[\frac{1}{p}\right]$, and $\alpha$ is again the dual action. In \cite{Latremoliere11c}, we computed the multipliers of $\Z\left[\frac{1}{p}\right]\times \Z\left[\frac{1}{p}\right]$ which is naturally homeomorphic to the solenoid group $\solenoid{p}$, and we classified all noncommutative solenoids up to *-isomorphism in terms of their multipliers, computing their K-theory.
\end{itemize}
In all our examples, any continuous length function would provide a {\qcms}. We will see later on that noncommutative solenoids and quantum tori can be seen as elements of the closure of the class of fuzzy tori, for our Gromov-Hausdorff propinquity, if we choose appropriately compatible continuous length functions.
\end{example}

A completely different set of examples is given by AF algebras.

\begin{example}[{\cite{Latremoliere15d}}]\label{AF-ex}
  Let $\A = \mathrm{cl}\left(\bigcup_{n\in\N}\A_n\right)$ be a C*-algebra which is the closure of an increasing union of a sequence of finite dimensional C*-subalgebras $(\A_n)_{n\in\N}$, i.e. an AF algebra. We assume that $\A$ carries a faithful tracial state $\mu$. For all $n\in\N$, there exists a unique conditional expectation $\mathds{E}_n : \A \twoheadrightarrow\A_n$ such that $\mu\circ\mathds{E}_n = \mu$. We then define, for all $a\in\A$:
  \begin{equation*}
    \Lip(a) = \sup\left\{ \frac{\norm{a-\mathds{E}_n(a)}{\A}}{\beta(n)} : n \in \N \right\}
  \end{equation*}
  allowing the value $\infty$, and with $(\beta(n))_{n\in\N} = (\frac{1}{\dim\A_n})_{n\in\N}$, or any choice of a sequence of positive numbers converging to $0$.

  We note that for all $a,b \in \A$, we have $\Lip(a b) \leq 2 ( \norm{a}{\A} \Lip(b) + \Lip(a) \norm{b}{\A} )$. So $(\A,\Lip)$ is a {\qcms}.
\end{example}

There are many other important examples of {\qcms s}, such as hyperbolic group C*-algebras \cite{Ozawa05}, nilpotent group C*-algebras \cite{Rieffel15b}, other quantum metrics on quantum tori \cite{Rieffel02, Latremoliere15c}, Podl{\`e}s spheres \cite{Kaad18}, various deformations of quantum metrics \cite{Latremoliere15b}, and more.

The notion of a {\lcqms} is more delicate to define, since the behavior of the {\MongeKant} is more subtle in this case \cite{Dobrushin70}. We introduced such a notion in \cite{Latremoliere05b,Latremoliere12b}, where we also study the noncommutative bounded-Lipschitz distance.

The class of {\qcms s} forms the class of objects of a category, which is helpful, for instance, in the appropriate definition for actions of groups and semigroups on {\qcms s}. There are several natural definitions of what a Lipschitz morphism ought to be. Maybe the one which is  at first sight the least demanding is the following.

\begin{definition}[{\cite{Latremoliere16b}}]
  Let $(\A,\Lip_\A)$ and $(\B,\Lip_\B)$ be two {\qcms s}. A \emph{Lipschitz morphism} $\pi : (\A,\Lip_\A) \rightarrow (\B,\Lip_\B)$ is a unital *-morphism $\pi : \A \rightarrow \B$ such that $\pi(\dom{\Lip_\A}) \subseteq \dom{\Lip_\B}$. 
\end{definition}

We now observe that in fact, there are important consequences to being a Lipschitz morphism, based on the following theorem:

\begin{theorem}[{\cite{Latremoliere16b}}]\label{equiv-thm}
Let $(\A,\Lip)$ be a quantum compact metric space, with $\Lip$ lower semi-continuous with domain $\dom{\Lip}$. Let $\mathsf{S}$ be a seminorm on $\dom{\Lip}$ such that:
\begin{enumerate}
\item $\mathsf{S}$ is lower semi-continuous with respect to $\|\cdot\|_\A$,
\item $\mathsf{S}(\unit_\A) = 0$.
\end{enumerate}
Then there exists $C > 0$ such that for all $a\in\dom{\Lip}$:
\begin{equation*}
\mathsf{S}(a) \leq C \Lip(a)\text{.}
\end{equation*}
\end{theorem}

We thus conclude, using Theorem (\ref{equiv-thm}) and \cite{Rieffel00}, that at least three natural notions of Lipschitz morphisms are indeed equivalent.
\begin{theorem}[{\cite{Rieffel99,Rieffel00,Latremoliere16b}}]
    Let $(\A,\Lip_\A)$ and $(\B,\Lip_\B)$ be two {\qcms s} and let $\pi : \A \rightarrow \B$ be a unital *-morphism. The following assertions are equivalent:
    \begin{enumerate}
      \item $\pi : (\A,\Lip_\A) \rightarrow (\B,\Lip_\B)$ is a Lipschitz morphism,
      \item there exists $k \geq 0$ such that $\Lip_\B\circ\pi \leq k \Lip_\A$,
      \item there exists $k \geq 0$ such that $\varphi \in \StateSpace(\B) \mapsto \varphi\circ\pi$ is a $k$-Lipschitz map from $(\StateSpace(\B),\Kantorovich{\Lip_\B})$ to $(\StateSpace(\A),\Kantorovich{\Lip_\A})$.
    \end{enumerate}
    Moreover, the real number $k$ in Assertion (2) and Assertion (3) can be chosen to be the same.

    Furthermore, the composition of two Lipschitz morphisms is a Lipschitz morphism.
\end{theorem}

We will find it useful also allow for the notion of a Lipschitz linear map, defined as follows:
\begin{definition}[{\cite{Latremoliere18b}}]\label{Lipschitz-linear-def}
  Let $(\A,\Lip_\A)$ and $(\B,\Lip_\B)$ be two {\qcms s}. A \emph{Lipschitz linear} map $\mu : (\A,\Lip_\A)\rightarrow(\B,\Lip_\B)$ is a positive unit-preserving linear map $\mu : \A\rightarrow\B$ for which there exists $k\geq 0$ such that $\Lip_\B\circ\mu\leq k \Lip_\A$.
\end{definition}
By Theorem (\ref{equiv-thm}), and with the notations of Definition (\ref{Lipschitz-linear-def}), $\mu : \A \rightarrow \B$ is Lipschitz linear if it is linear, positive, unital, and $\mu(\dom{\Lip_\A})\subseteq\dom{\Lip_\B}$.

In general, there is a natural notion of \emph{dilation} (or Lipschitz constant) for Lipschitz morphisms and more generally, Lipschitz linear maps. This notion will be useful in formulating equicontinuity condition later on when working with actions of groups and semigroups.

\begin{notation}[{\cite{Latremoliere16b}}]
  Let $(\A,\Lip_\A)$ and $(\B,\Lip_\B)$ be two {\qcms s}. If $\pi : \A \rightarrow \B$ is a unital positive linear map, then
  \begin{equation*}
    \dil{\pi} = \inf\left\{ k > 0 : \forall a \in \sa{\A} \quad \Lip\circ\pi(a) \leq k \Lip(a) \right\} \text{.}
  \end{equation*}
  By definition, $\dil{\pi} < \infty$ if and only if $\pi$ is a Lipschitz linear map.
\end{notation}

In \cite{Latremoliere16b}, we use the notion of dilation for Lipschitz morphisms to define the noncommutative version of the Lipschitz distance. The Lipschitz distance dominates the dual propinquity we will review in the next section, and in fact, closed balls for the Lipschitz distance are compact for the dual propinquity. Of course, as in the classical picture, the Lipschitz distance is only interesting between *-isomorphic C*-algebras endowed with various metrics, so it is much too strong for most of our purpose.

Another useful tool when working with actions via Lipschitz morphisms will be a metric on the space of *-morphisms, or more generally unital positive maps, induced by L-seminorms. The following result generalizes slightly the last statement of \cite{Latremoliere16b}.

\begin{theorem}[{\cite{Latremoliere16b}}]
  Let $(\A,\Lip_\A)$ be a {\qcms} and let $\B$ be a unital C*-algebra. If for any two unital linear maps $\alpha$, $\beta$ from $\A$ to $\B$, we set
  \begin{equation*}
    \KantorovichDist{\Lip_\A}{\alpha}{\beta} = \sup\left\{ \norm{\alpha(a) - \beta(a)}{\B} : a\in\dom{\Lip}, \Lip(a) \leq 1 \right\} \text{,}
  \end{equation*}
  then $\mathrm{mkD}_{\Lip_\A}$ is a distance on the space $\mathcal{B}_1(\A,\B)$ of unit preserving bounded linear maps, which, on any norm-bounded subset, metrizes the initial topology induced by the family of seminorms
  \begin{equation*}
    \left\{ \alpha \in \mathcal{B}_1 \mapsto \norm{\alpha(a)}{\B} : a \in \A  \right\} \text{.}
  \end{equation*}
\end{theorem}

\begin{proof}
  Let $(\alpha_n)_{n\in\N}$ be a sequence of unit preserving linear maps converging to some unital linear map $\alpha_\infty$ for $\KantorovichDist{\Lip_\A}{}{}$, and for which there exists some $B > 0$ such that for all $n\in\N\cup\{\infty\}$, we have $\opnorm{\alpha_n}{\A}{\B} \leq B$, where $\opnorm{\cdot}{\A}{\B}$ is the operator norm for linear maps from $\A$ to $\B$. Let $a\in \sa{\A}$ and $\varepsilon > 0$. Since $\dom{\Lip_\A}$ is dense in $\sa{\A}$, there exists $a' \in \dom{\Lip_\A}$ such that $\norm{a-a'}{\A} < \frac{\varepsilon}{3 B}$. By definition, there exists $N\in\N$ such that for all $n\geq N$, we have $\KantorovichDist{\Lip_\A}{\alpha_n}{\alpha_\infty} < \frac{\varepsilon}{3(\Lip_\A(a')+1)}$. Thus
  \begin{equation*}
    \norm{\alpha_\infty(a)-\alpha_n(a)}{\B} \leq \norm{\alpha_n(a - a')}{\B} + \norm{\alpha_n(a')-\alpha_\infty(a')}{\B} + \norm{\alpha_\infty(a-a')}{\B} < \varepsilon \text{.}
  \end{equation*}
  Thus for all $a\in\sa{\A}$, the sequence $(\alpha_n(a))_{n\in\N}$ converges to $\alpha_\infty(a)$. By linearity, we then conclude $(\alpha_n(a))_{n\in\N}$ converge to $\alpha_\infty(a)$ for $\norm{\cdot}{\B}$.

  Conversely, assume that for all $a\in\A$, the sequence $(\alpha_n(a))_{n\in\N}$ converges  to $\alpha_\infty(a)$ in $\B$, and again assume that there exists $B > 0$ such that for all $n\in\N\cup\{\infty\}$, we have $\opnorm{\alpha_n}{\A}{\B} \leq B$. Let $\varepsilon > 0$ and fix $\mu \in \StateSpace(\A)$. As $\Lip_\A$ is a L-seminorm, $L = \{a\in\sa{\A}:\Lip_\A(a)\leq 1,\mu(a) = 0\}$ is totally bounded. Thus, there exists a finite $\frac{\varepsilon}{3 B}$-dense set $F \subseteq L$ of $L$. As $F$ is finite, by assumption, there exists $N\in\N$ such that for all $n\geq N$ and all $a\in F$, we have $\norm{\alpha_n(a)-\alpha_\infty(a)}{\B} < \frac{\varepsilon}{3}$. If $n\geq N$ and $a\in \sa{\A}$ such that $\Lip_\A(a)\leq 1$ then there exists $a' \in F$ such that $\norm{a-\mu(a)\unit_\A - a'}{\A} < \frac{\varepsilon}{3 B}$, and thus
\begin{align*}
  \norm{\alpha_n(a)-\alpha_\infty(a)}{\B} 
  &\leq \norm{\alpha_n(a-\mu(a)\unit_\A) - \alpha_\infty(a-\mu(a)\unit_\A)}{\B} \\
  &\leq \norm{\alpha_n(a-\mu(a)\unit_\A - a')}{\B} \\
  &\quad + \norm{\alpha_n(a')-\alpha_\infty(a')}{\B} + \norm{\alpha_\infty(a-\mu(a)\unit_\A)-a'}{\B} \\
  &\leq B\frac{\varepsilon}{3B} + \frac{\varepsilon}{3} + B\frac{\varepsilon}{3B} < \varepsilon \text{.}
\end{align*}
Thus for $n\geq N$, we have $\KantorovichDist{\Lip_\A}{\alpha_n}{\alpha_\infty} < \varepsilon$.
\end{proof}

There are of course at least four common notions of morphisms over the category of metric spaces: continuous functions, uniformly continuous functions, Lipschitz functions, and isometries. Continuous functions correspond to *-morphism of course. As we work with compact metric space, uniform continuity and continuity are equivalent. We just define the notion of Lipschitz morphism. For our work, we also will need to understand what a quantum isometry should be.

Rieffel observed that McShane's theorem \cite{McShane34} on extension of \emph{real-valued} Lipschitz functions can be used to characterize isometries. The emphasis on real-valued Lipschitz functions, rather than complex valued, means for our purpose that the isometric property will only involve self-adjoint elements. The definition of a quantum isometry is thus given by:

\begin{definition}[{\cite{Rieffel99,Rieffel00}}]
  Let $(\A,\Lip_\A)$ and  $(\B,\Lip_\B)$ be {\qcms s}.
  \begin{itemize}
  \item A \emph{quantum isometry} $\pi : (\A,\Lip_\A) \rightarrow (\B,\Lip_\B)$ is a *-epimorphism from $\A$ onto $\B$ such that for all $b \in \dom{\Lip_\A}$
    \begin{equation*}
      \Lip_\B(b) = \inf\left\{ \Lip_\A(a) : \pi(a) = b \right\} \text{.}
    \end{equation*}
  \item A \emph{full quantum isometry} $\pi : (\A,\Lip_\A) \rightarrow (\B,\Lip_\B)$ is a *-isomorphism from $\A$ onto $\B$ such that $\Lip_\B\circ\pi = \Lip_\A$.
  \end{itemize}
\end{definition}

We observe that in \cite{Rieffel00}, Rieffel proved that quantum isometries can be chosen as morphisms for a category over {\qcms s}; this is a subcategory of the category of {\qcms s} with Lipschitz morphisms (as quantum isometries are obviously $1$-Lipschitz morphisms).

The notion of full quantum isometry is essential to our work: it is the notion which we take as the basic equivalence between {\qcms s}, i.e. two {\qcms s} are, for our purpose, the same when they are fully quantum isometric.

\section{The Gromov-Hausdorff Propinquity}

The Gromov-Hausdorff distance \cite{Gromov,Gromov81} is a complete metric up to full isometry on the class of proper metric spaces, initially described by Edwards for compact metric spaces \cite{Edwards75}. This metric is an intrinsic version of the Hausdorff distance \cite{Hausdorff}. It is constructed by taking the infimum of the Hausdorff distance between any two isometric copies of given compact metric spaces. By duality, as we have a notion of quantum isometry, we obtain a notion of something we shall call a tunnel between two {\qcms s}.

\begin{definition}[{\cite{Latremoliere13b}}]
  Let $(\A_1,\Lip_1)$ and $(\A_2,\Lip_2)$ be two {$F$-\qcms s} for a permissible function $F$. An \emph{$F$-tunnel} $\tau = (\D,\Lip,\pi_1,\pi_2)$ from $(\A_1,\Lip_1)$ to $(\A_2,\Lip_2)$ is a {\Qqcms{F}} $(\D,\Lip)$ and for each $j\in\{1,2\}$, a quantum isometry $\pi_j : (\D,\Lip)\twoheadrightarrow(\A_j,\Lip_j)$. The space $(\A_1,\Lip_1)$ is the \emph{domain} $\dom{\tau}$ of $\tau$,  while the space $(\A_2,\Lip_2)$ is the \emph{codomain} $\codom{\tau}$ of $\tau$.
\end{definition}

      \begin{figure}[h]
        \begin{equation*}
          \xymatrix{
            & (Z,\mathrm{d}_Z)  & \\
            (X,\mathrm{d}_X) \ar@{^{(}->}[ur]^{\iota_X} & & (Y,\mathrm{d}_Y)  \ar@{_{(}->}[ul]_{\iota_Y}
          }
        \end{equation*}
        \caption{Isometric Embeddings}\label{GH-fig}
      \end{figure}
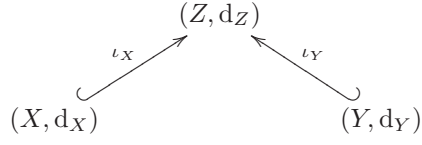
      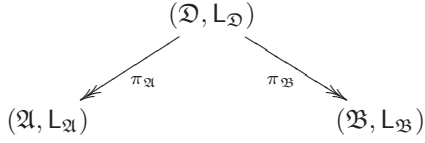
\begin{figure}[h]
        \begin{equation*}
          \xymatrix{
            & (\D,\Lip_\D)  \ar@{->>}[dl]^{\pi_\A} \ar@{->>}[dr]_{\pi_\B} & \\
            (\A,\Lip_\A)  & & (\B,\Lip_\B)
          }
        \end{equation*}
        \caption{A tunnel}\label{tunnel-fig}
      \end{figure}

\begin{figure}[h]
\begin{equation*}
\xymatrix{
 & (\StateSpace(\D),\Kantorovich{\Lip_\D})  & \\
(\StateSpace(\A),\Kantorovich{\Lip_\A}) \ar@{^{(}->}[ur]^{\pi_\A^\ast} & & (\StateSpace(\B),\Kantorovich{\Lip_\B})  \ar@{_{(}->}[ul]_{\pi_\B^\ast}
}
\end{equation*}
\caption{Isometric Embeddings of state spaces induced by tunnels}\label{GH-fig}
\end{figure}
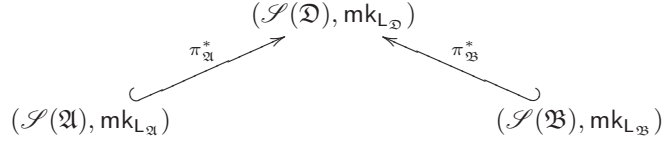

We need to associate a quantity to any given tunnel which estimates how far apart its domain and codomain are. The following quantity is the choice which gives us the best construction of our propinquity \cite{Latremoliere14}, though a certain alternative method can be found in \cite{Latremoliere13b} where the dual propinquity was originally developed.

\begin{definition}[{\cite{Latremoliere14}}]
  The \emph{extent} $\tunnelextent{\tau}$ of a tunnel $\tau = (\D,\Lip,\pi_1,\pi_2)$ is given as the real number
  \begin{equation*}
    \max_{j\in\{1,2\}} \Haus{\Kantorovich{\Lip}}\left(\StateSpace(\D), \left\{ \varphi\circ\pi_j : \varphi \in \StateSpace(\A_j) \right\} \right) \text{.}
  \end{equation*}
\end{definition}

Tunnels between two {\qcms s} involve a third one, and we have to choose what properties, if any, this third {\qcms} should possess. The most important restriction we must impose for our theory to work is a particular choice of permissible function, i.e. of a quasi-Leibniz inequality satisfied by all three L-seminorms. The key is that the choice must be uniform throughout our construction of the propinquity, so that the propinquity is indeed a metric up to full quantum isometry. The quasi-Leibniz inequality is used to obtain the multiplicative property.

There may be situations which require more properties for L-seminorms. An example is the strong Leibniz property introduced by Rieffel in \cite{Rieffel10c}. Thus, it is helpful to keep some level of generality in our construction by allowing flexibility in restricting the class of tunnels used, beyond the quasi-Leibniz restriction. In order for our construction to lead to a metric, we ask that a chosen class of tunnel meet the following definition.

\begin{definition}[{\cite{Latremoliere14}}]
  Let $F$ be a permissible function. A class $\mathcal{T}$ of $F$-tunnels is \emph{appropriate} for a nonempty class $\mathcal{C}$ of {$F$-{\qcms s}} when
  \begin{enumerate}
  \item for every $\tau\in\mathcal{T}$, we have $\dom{\tau},\codom{\tau} \in \mathcal{C}$,
  \item for every $\mathds{A},\mathds{B} \in \mathcal{C}$, there exists $\tau \in \mathcal{T}$ with domain $\mathds{A}$ and codomain $\mathds{B}$,
  \item if $\tau = (\D,\Lip,\pi,\rho) \in \mathcal{T}$ then $\tau^{-1} = (\D,\Lip,\rho,\pi) \in \mathcal{T}$,
  \item if $\varepsilon > 0$, and if $\tau_1,\tau_2 \in \mathcal{T}$ with $\codom{\tau_1} = \dom{\tau_2}$, then there exists $\tau\in\mathcal{T}$ with $\dom{\tau} = \dom{\tau_1}$, $\codom{\tau} = \codom{\tau_2}$, and $\tunnelextent{\tau} \leq \tunnelextent{\tau_1} + \tunnelextent{\tau_2} + \varepsilon$,
  \item if $(\A,\Lip_\A)$,$(\B,\Lip_\B)$ are in $\mathcal{C}$ and if there exists a full quantum isometry $\pi : (\A,\Lip_\A)\rightarrow (\B,\Lip_\B)$, then $(\A,\Lip_\A,\pi,\mathrm{id}) \in \mathcal{T}$, with $\mathrm{id}$ the identity *-automorphism of $\A$.
  \end{enumerate}
\end{definition}

We proved in \cite{Latremoliere14} that the class of all $F$-tunnels is appropriate for the class of all {$F$-\qcms s}, and this is, for our own work, the sort of class we work with (note that we must impose the restriction to work with tunnels which all share the same quasi-Leibniz inequality, as parameterized by $F$).

\begin{notation}
  Let $\mathcal{T}$ be a class of tunnels appropriate for a nonempty class of $F$-{\qcms s} $\mathcal{C}$. The set of all $F$-tunnels in $\mathcal{T}$ from $(\A,\Lip_\A)$ to $(\B,\Lip_\B)$, both chosen in $\mathcal{C}$, is denoted by
  \begin{equation*}
    \tunnelsetc{\A,\Lip_\A}{\B,\Lip_\B}{\mathcal{T}} \text{.}
  \end{equation*}
  When $\mathcal{T}$ is simply the class of all $F$-tunnels, we then write the class of all $F$-tunnels from $(\A,\Lip_\A)$ to $(\B,\Lip_\B)$, for any two $F$-{\qcms s} $\mathds{A}$, $\mathds{B}$, by
  \begin{equation*}
        \tunnelsetc{\A,\Lip_\A}{\B,\Lip_\B}{F} \text{.}
  \end{equation*}
\end{notation}

The dual Gromov-Hausdorff propinquity is now constructed using the same technique as Edwards' and Gromov's. It enjoys many good properties, among which is being a metric up to full quantum isometry.

\begin{theorem-definition}[{\cite{Latremoliere13,Latremoliere13b,Latremoliere14,Latremoliere15}}]
  Let $\mathcal{T}$ be a class of tunnels appropriate for a nonempty class of {$F$-\qcms s} for some permissible function $F$. We define the \emph{dual $\mathcal{T}$-propinquity} between any two {\qcms s} $(\A,\Lip)$ and $(\B,\Lip)$ in $\mathcal{C}$ as the real number
  \begin{equation*}
    \dpropinquity{\mathcal{T}}((\A,\Lip_\A),(\B,\Lip_\B)) = \inf\left\{ \tunnelextent{\tau} : \tau \in \tunnelsetc{\A,\Lip_\A}{\B,\Lip_\B}{\mathcal{T}} \right\}\text{.}
  \end{equation*}

  The dual propinquity is a metric up to quantum full isometry on $\mathcal{C}$, i.e. it is a pseudo-metric on $\mathcal{C}$ such that any $\mathds{A},\mathds{B} \in \mathcal{C}$ are fully quantum isometry if and only if $\dpropinquity{\mathcal{T}}(\mathds{A},\mathds{B}) = 0$.

  Moreover, if $(X,d_X)$ and $(Y,d_Y)$ are two compact metric spaces and if $\Lip_X$ and $\Lip_Y$ are the Lipschitz seminorms induced respectively on $C(X)$ by $d_X$ and $C(Y)$ by $d_Y$, and if $\mathrm{GH}$ is the usual Gromov-Hausdorff distance, then
  \begin{equation*}
    \mathrm{GH}((X,d_X),(Y,d_Y)) \leq \dpropinquity{\mathcal{T}}( (C(X),\Lip_X), (C(Y),\Lip_Y)) \text{,}
  \end{equation*}
as long as $(C(X),\Lip_X),(C(Y),\Lip_Y) \in \mathcal{C}$, and the topology induced on the class of classical metric spaces in $\mathcal{C}$ via $\dpropinquity{T}$ is the same as the topology induced by the Gromov-Hausdorff distance.

Last, let us denote by $\dpropinquity{F}$ the dual propinquity induced on all $F$-{\qcms s} using all possible $F$-tunnels. Then $\dpropinquity{F}$ is complete if $F$ is continuous.
\end{theorem-definition}

Thus in particular, $\dpropinquity{F}$ is a complete metric up to full quantum isometry which induces the same topology as the Gromov-Hausdorff distance on classical compact metric spaces. This is the main tool for our research.

We note that completeness is a desirable property for obvious reasons, including the study of compactness. We derive an analogue of Gromov's compactness theorem in \cite{Latremoliere15}.

Now, in order to actually prove convergence results for the dual propinquity, it is of course desirable to have a prolific source of tunnels for any two {\qcms s}. We actually first discovered this source in \cite{Latremoliere13} before we introduced tunnels \cite{Latremoliere13b}. The idea behind a bridge could be said to be a far-reaching generalization of the idea of an intertwiner between two *-representations of a C*-algebra, though we now work with representations of two different C*-algebras, restricting ourselves to faithful unital representations on the same space, and we will learn to measure how good of an ``approximate intertwiner'' a particular bridge is. This informal idea begins with the following definition.

\begin{definition}[{\cite{Latremoliere13}}]
  Let $\A_1$ and $\A_2$ be two unital C*-algebras. A \emph{bridge} $\gamma = (\D,x,\pi_1,\pi_2)$ is a given by
  \begin{enumerate}
  \item a unital C*-algebra $\D$,
  \item for each $j\in \{1,2\}$, a unital *-monomorphism $\pi_j : \A_j \hookrightarrow \D$,
  \item an element $x\in \D$, called the \emph{pivot} of $\gamma$, for which there exists a state $\varphi \in \StateSpace(\D)$ such that $\varphi((1-x)^\ast(1-x)) = \varphi((1-x)(1-x)^\ast) = 0$.
  \end{enumerate}
  The \emph{domain} $\dom{\gamma}$ of $\gamma$ is $\A_1$ while the codomain $\codom{\gamma}$ of $\gamma$ is $\A_2$.
\end{definition}

To measure how far apart the domain and codomain of a bridge are, the following two objects which arise immediately from our definition will be helpful.

\begin{notation}[{\cite{Latremoliere13}}]
  Let $\gamma = (\D,x,\pi_\A,\pi_\B)$ be a bridge from $\A$ to $\B$, where $\A$ and $\B$ are unital C*-algebras. The \emph{$1$-level set} of $x$ in $\D$ is the set of states
  \begin{equation*}
    \StateSpace_1(\D|x) = \left\{ \varphi \in \StateSpace(\D) \middle\vert \varphi((1-x)^\ast(1-x)) = \varphi((1-x)(1-x)^\ast) = 0  \right\} \text{.}
  \end{equation*}

  Moreover, we define a seminorm on $\A\oplus\B$ by setting for all $a\in\A$, $b\in \B$:
  \begin{equation*}
    \bridgenorm{\gamma}{a,b} = \norm{\pi_\A(a) x - x \pi_\B(b)}{\D} \text{.}
  \end{equation*}
\end{notation}

We now associate a number to our bridge. In the following definition, it may be helpful to think of the height as measuring how far the pivot is from the identity in a manner employing the quantum metrics. On the other hand, the reach measures how far apart the domain and the codomain are using our almost intertwiner, the pivot, and the seminorm it defines.  Importantly, all the quantities used to quantify a bridge involve the quantum metrics, which are not used in defining the bridge itself.

\begin{definition}[{\cite{Latremoliere13}}]
  Let $(\A_1,\Lip_1)$ and $(\A_2,\Lip_2)$ be two {\qcms}. If $\gamma = (\D,x,\pi_1,\pi_2)$ is a bridge from $\A_1$ to $\A_2$ then:
  \begin{enumerate}
    \item the \emph{height} $\bridgeheight{\gamma}{\Lip_\A,\Lip_\B}$ of $\gamma$ is the real number
      \begin{equation*}
        \max_{j\in\{1,2\}} \Haus{\Kantorovich{\Lip_j}}\left( \StateSpace(\A_j), \left\{ \varphi\circ\pi_j : \varphi \in \StateSpace_1(\D|x) \right\} \right) \text{,}
      \end{equation*}
      \item the \emph{reach} $\bridgereach{\gamma}{\Lip_1,\Lip_2}$ of $\gamma$ is the real number
        \begin{equation*}
          \Haus{\mathrm{bn}_\gamma}\left( \left\{(a,0):a\in\sa{\A_1},\Lip_1(a)\right\}, \left\{(0,b):b\in\sa{\A_2},\Lip_2(b)\leq 1\right\} \right) \text{,}
        \end{equation*}
      \item the \emph{length} $\bridgelength{\gamma}{\Lip_1,\Lip_2}$ of $\gamma$ is $\max\left\{\bridgeheight{\gamma}{\Lip_1,\Lip_2}, \bridgereach{\gamma}{\Lip_1,\Lip_2}\right\}$.
  \end{enumerate}
\end{definition}

We now see that bridges provide a mean to build tunnels. This is how most  non-trivial tunnels are constructed in our work so far. Maybe a main reason for this fact is that constructing L-seminorms is usually delicate, but bridges provide such L-seminorms in a systematic manner and moreover, their length gives an estimate on the extent of the resulting tunnel.

\begin{theorem}[{\cite{Latremoliere13,Latremoliere13b}}]\label{tunnel-from-bridge-thm}
  Let $(\A,\Lip_\A)$ and $(\B,\Lip_\B)$ be two {$F$-\qcms s} for some permissible function $F$. If $\gamma = (\D,x,\pi_\A,\pi_\B)$ is a bridge from $\A$ to $\B$, if $\varepsilon \geq 0$ is chosen so that $\lambda = \bridgelength{\gamma}{\Lip_\A,\Lip_\B} + \varepsilon > 0$, and if, for all $(a,b) \in \sa{\A}\oplus\sa{\B}$, we set
  \begin{equation*}
    \Lip(a,b) = \max\left\{ \Lip_\A(a), \Lip_\B(b), \frac{1}{\lambda}\bridgenorm{\gamma}{a,b} \right\}\text{,}
  \end{equation*}
  then $\left(\A\oplus\B,\Lip,\rho_\A,\rho_\B\right)$, where $\rho_\A:(a,b)\in\A\oplus\B\mapsto a$ and $\rho_\B:(a,b)\in\A\oplus\B\mapsto b$, is an $F$-tunnel from $(\A,\Lip_\A)$ to $(\B,\Lip_\B)$, of extent at most $\lambda$.

In particular
\begin{equation*}
  \dpropinquity{F}((\A,\Lip_\A),(\B,\Lip_\B)) \leq \bridgelength{\gamma}{\Lip_\A,\Lip_\B} \text{.}
\end{equation*}
\end{theorem}

We used the construction of appropriate bridges to prove the following examples of convergence for the dual propinquity.

\begin{example}[{\cite{Latremoliere13b}}]
  Let $\ell$ be a continuous length function on $\T^d$. For any $G\subseteq \T^d$ a closed subgroup and $\sigma$ a multiplier of the Pontryagin dual $\widehat{G}$ of $G$, for any $a\in C^\ast(\widehat{G},\sigma)$, we set as in Example (\ref{group-lip-ex}):
  \begin{equation*}
    \Lip_{G,\sigma}(a) = \sup\left\{\frac{\|\alpha^g(a)-a\|_{C^\ast(\widehat{G},\sigma)}}{\ell(g)} : g\in G\setminus\{1\} \right\}
    \end{equation*}
    where $\alpha$ is the dual action of $G$ on $C^\ast(\widehat{G},\sigma)$.
    
    If $(G_n)_{n\in\N}$ is a sequence of closed subgroups of $\T^d$ converging to $\T^d$ for the Hausdorff distance $\Haus{\ell}$, and if $(\sigma_n)_{n\in\N}$ is a sequence of multipliers of $\Z^d$ converging pointwise to some $\sigma$, with $\sigma_n(g)=1$ if $g$ is the coset of $0$ for $\widehat{G_n}$ , then:
    \begin{equation*}
      \lim_{n\rightarrow\infty} \dpropinquity{}((C^\ast(\widehat{G_n},\sigma_n),\Lip_{\widehat{G_n},\sigma_n}),(C^\ast(\Z^d,\sigma),\Lip_{\Z^d,\sigma})) = 0 \text{.}
    \end{equation*}

    In particular, the function which maps a multiplier to a quantum  torus is continuous for the dual  propinquity, and quantum tori are limits of fuzzy tori.
\end{example}

\begin{example}[{\cite{Latremoliere16}}]
  Noncommutative solenoids are limits of quantum tori, and consequently, limits of fuzzy tori, for the dual propinquity, for the appropriate choice of a metric on the solenoid groups.
\end{example}

\begin{example}[{\cite{Latremoliere15d}}]
  We use the same notation as in Example (\ref{AF-ex}). We then have:
  \begin{itemize}
   \item $(\A,\Lip) = \dpropinquity{}-\lim_{n\rightarrow\infty}(\A_n,\Lip)$,
   \item the natural map from the Baire space to UHF algebras is Lipschitz.
\end{itemize}

  Another example is given by the Effros-Shen algebras \cite{Effros80b}. For $\theta \in (0,1)\setminus\Q$, let $\theta = \lim_{n\rightarrow\infty}\frac{p_n^\theta}{q_n^\theta}$ with $\frac{p_n^\theta}{q_n^\theta} = \cfrac{1}{a_1 + \cfrac{1}{a_2 + \cfrac{1}{\ddots + \frac{1}{a_n}}}}$ for $a_1,\ldots\in\N$.
   
  Set $\alg{AF}_\theta = \varinjlim_{n\rightarrow\infty} \left(\alg{M}_{q_n} \oplus \alg{M}_{q_{n-1}}, \psi_{n,\theta} \right)$ where $\psi_{n,\theta}$ is defined by:
  \begin{equation*}
    (a,b) \in \alg{M}_{q_n}\oplus\alg{M}_{q_{n-1}} \mapsto \left( \begin{pmatrix} a & & & \\ & \ddots & & \\ & & a & \\ & & & b\end{pmatrix}, a   \right)  \text{.}
  \end{equation*}
    
  Let $\Lip_\theta$ the L-seminorm for this data as in Example (\ref{AF-ex}).
      For all $\theta\in(0,1)\setminus\Q$, we have:
      \begin{equation*}
        \lim_{\substack{\vartheta\rightarrow\theta\\ \vartheta\not\in\Q}} \dpropinquity{}((\alg{AF}_\vartheta,\Lip_\vartheta),(\alg{AF}_\theta,\Lip_\theta)) = 0 \text{.}
      \end{equation*}

\end{example}

A final remark in this section concerns the following question: could we choose, in the construction of the dual propinquity, only those tunnels which emerge from bridges as in Theorem (\ref{tunnel-from-bridge-thm})? We do not know, nor believe, that tunnels obtained from bridges form an appropriate class with the class of all $F$-{\qcms s} for any particular permissible function $F$. The problem can be summed up by saying that given two bridges $\gamma_1$ and $\gamma_2$ with $\codom{\gamma_1} = \dom{\gamma_2}$, we do not know how to build a single bridge whose length is approximately the sum of the lengths of $\gamma_1$ and $\gamma_2$, going from $\dom{\gamma_1}$ to $\codom{\gamma_2}$. This, in turn, means that our construction of the propinquity would fail to satisfy the triangle inequality. This issue was nontrivial and is the tip of the iceberg regarding the difficulties of working with Leibniz or quasi-Leibniz seminorms. 

However, \cite[Theorem 3.1]{Latremoliere14} does essentially show us how to take the ``closure'' of the class of all tunnels-from-bridges to obtain an appropriate class for all {\qcms s}. It turns out that in the case of tunnels constructed from bridges, \cite[Theorem 3.1]{Latremoliere14} can be seen as introducing between $\gamma_1$ and $\gamma_2$, as above, a very short bridge from $\codom{\gamma_1}$ to $\dom{\gamma_2}$. Thus, \cite[Theorem 3.1]{Latremoliere14} dictates that we really want to work, not just with bridges, but with tunnels built by a finite collection of bridges, each ending where the next starts. The very short in-between bridges can in fact be taken so short as to have length zero in this case, and disappear---leaving us with the idea of \emph{treks} which we used, in our first work on the propinquity \cite{Latremoliere13}, to define a first form of the propinquity called the quantum propinquity.

We note that the length of a bridge dominates the extent of its canonically associated tunnel by Theorem (\ref{tunnel-from-bridge-thm}), so even reconciling the idea of treks with the idea of almost composition of tunnels as in \cite[Theorem 3.1]{Latremoliere14} does not mean the quantum propinquity, which uses the length of bridges directly rather than the extent of the associated tunnels, is equal to the dual propinquity---the former still dominates the latter as far as we can tell. But the two pictures are now closer together.

As is seen for instance with quantum tori, group actions can be used as a means of transport of structure to define a noncommutative geometry. Semigroups, or rather semigroups of completely positive unital maps, can be interpreted as a form of noncommutative heat semigroups, and give rise to a differential calculus where the generator of the semigroup is a noncommutative Laplacian. Both are very interesting and important models for noncommutative geometry. More generally, symmetries have been a central concept in the development of mathematical physics models in particle physics, so, with our motivation for this research program  in mind, we want to understand: what is the interplay between symmetries as encoded in group actions, and dynamics encoded as group or semigroup actions, and the dual propinquity, and our approach to noncommutative metric geometry? There are actually some rather pleasant facts regarding these matters, and we begin by describing how to capture group and semigroup actions in our metric framework---only to see later a nice compactness-type result which shows that the dual propinquity is keen to remember some symmetries of spaces, under natural non-degeneracy conditions.

\section{The Covariant Propinquity}

We begin by defining the objects on which we are going to define an extension of the dual propinquity. The idea is to bring together a metrized group or semigroup, a {\qcms}, and an action of the former on the latter by Lipschitz  morphisms, or at least Lipschitz linear maps.

We thus begin with:

\begin{definition}
  A \emph{proper monoid} $(G,\delta)$ is a monoid (i.e. a set $G$ endowed with an associative binary operation and with an identity element) and a left-invariant distance $\delta$ on $G$ whose closed balls are all compact.

  For any two proper monoids $(G,\delta_G)$ and $(H,\delta_H)$, a \emph{proper monoid morphism} $\pi : (G,\delta_G) \rightarrow (H,\delta_H)$ is a map from $G$ to $H$ such that:
  \begin{itemize}
    \item $\pi$ maps the identity element of $G$ to the identity element of $H$,
    \item $\forall g,h \in G \quad \pi(g h) = \pi(g) \pi(h)$,
    \item $\pi$ is continuous.
  \end{itemize}
\end{definition}

\begin{definition}[{\cite{Latremoliere18b}}]
  Let $F$ be a permissible function. A \emph{Lipschitz dynamical $F$-system} $(\A,\Lip,G,\delta,\alpha)$ is given by:
\begin{enumerate}
  \item an $F$-\qcms $(\A,\Lip)$,
  \item a proper monoid $(G,\delta)$,
  \item a strongly continuous action $\alpha$ of $G$ on $\A$: for all $a\in\A$ and $g \in G$, we have:
    \begin{equation*}
      \lim_{h\rightarrow g} \norm{\alpha^h(a) - \alpha^g(a)}{\A} = 0\text{,}
    \end{equation*}
  \item $g\in G\mapsto \dil{\alpha^g}$ is locally bounded: for all $\varepsilon>0$ and $g\in G$ there exist $D > 0$ and a neighborhood $U$ of $g$ in $G$ such that if $h\in U$ then $\dil{\alpha^h} \leq D$.
\end{enumerate}

A \emph{Lipschitz $C^\ast$-dynamical $F$-system} $(\A,\Lip,G,\delta,\alpha)$ is a Lipschitz dynamical system where $G$ is a proper group and $\alpha^g$ is a Lipschitz unital *-automorphism for all $g \in G$.
\end{definition}

We now wish to endow classes of Lipschitz dynamical systems with a sort of dual propinquity. To this end, we first must understand how to define a Gromov-Hausdorff distance for proper monoids. We propose the following construction, which encompasses natural ideas, but in a manner which defines a nice metric.

\begin{notation}
  For a metric space $(X,\delta)$, if $x\in X$ and $r\geq 0$, then the closed ball in $(X,\delta)$ centered at $x$, of radius $r$, is denoted $X_\delta[x,r]$, or simply $X[x,r]$.
  If $(G,\delta)$ is a metric monoid with identity element $e \in G$, and if $r \geq 0$, then $G[e,r]$ is denoted by $G[r]$. 
\end{notation}

We define our distance between two proper metric monoids $(G_1,\delta_1)$ and $(G_2,\delta_2)$ by measuring how far a given pair of maps $\varsigma_1:G_1\rightarrow G_2$ and $\varsigma_2 : G_2\rightarrow G_1$ are from being an isometric isomorphism and its inverse. 

\begin{definition}[{\cite{Latremoliere18b}}]\label{almost-iso-def}
Let $(G_1,\delta_1)$ and $(G_2,\delta_2)$ be two metric monoids with identity elements $e_1$ and $e_2$. An \emph{$r$-local $\varepsilon$-almost isometric isomorphism} $(\varsigma_1,\varsigma_2)$, for $\varepsilon \geq 0$ and $r \geq 0$,  is an ordered pair of maps $\varsigma_1 : G_1[r] \rightarrow G_2$ and $\varsigma_2 : G_2[r] \rightarrow G_1$ such that for all $\{j,k\} = \{1,2\}$:
\begin{equation*}
\forall g,g' \in G_j[r] \quad \forall h \in G_k[r] \quad \left| \delta_k(\varsigma_j(g)\varsigma_j(g'),h) - \delta_j(gg',\varsigma_k(h))\right| \leq \varepsilon\text{,}
\end{equation*}
and
\begin{equation*}
\varsigma_j(e_j) = e_k \text{.}
\end{equation*}
 The set of all $r$-local $\varepsilon$-almost isometric isomorphism is denoted by:
\begin{equation*}
  \UIso{\varepsilon}{(G_1,\delta_1)}{(G_2,\delta_2)}{r} \text{.}
\end{equation*}
\end{definition}

Our covariant Gromov-Hausdorff distance over the class of proper metric monoids is then defined along the lines of Gromov's distance. The bound $\frac{\sqrt{2}}{2}$ is just here to ensure that our metric satisfies the triangle inequality. Our construction follows Gromov's insight on how to define an intrinsic Hausdorff distance between pointed, proper spaces---rather than the Edwards definition we used for {\qcms s}---where we chose as base point the identity elements.

\begin{definition}[{\cite{Latremoliere18b}}]\label{group-GH-def}
The \emph{Gromov-Hausdorff monoid distance} $\Upsilon((G_1,\delta_1),(G_2,\delta_2))$ between two proper metric monoids $(G_1,\delta_1)$ and $(G_2,\delta_2)$ is given by:
  \begin{multline*}
    \Upsilon((G_1,\delta_1),(G_2,\delta_2)) = \\
    \min\left\{ \frac{\sqrt{2}}{2}, \inf\left\{ \varepsilon > 0 \middle\vert \UIso{\varepsilon}{(G_1,\delta_1)}{(G_2,\delta_2)}{\frac{1}{\varepsilon}} \not= \emptyset \right\}\right\} \text{.}
  \end{multline*}
\end{definition}

We then record:
\begin{theorem}[{\cite{Latremoliere18b}}]\label{upsilon-metric-thm}
For any proper metric monoids $(G_1,\delta_1)$, $(G_2,\delta_2)$ and $(G_3,\delta_3)$:
\begin{enumerate}
\item $\Upsilon((G_1,\delta_1),(G_2,\delta_2)) \leq \frac{\sqrt{2}}{2}$,
\item $\Upsilon((G_1,\delta_1),(G_2,\delta_2)) = \Upsilon((G_2,\delta_2),(G_1,\delta_1))$,
\item $\Upsilon((G_1,\delta_1),(G_3,\delta_3)) \leq \Upsilon((G_1,\delta_1),(G_2,\delta_2)) + \Upsilon((G_2,\delta_2),(G_3,\delta_3))$,
\item If $\Upsilon((G_1,\delta_1),(G_2,\delta_2)) = 0$ if and only if there exists a monoid isometric isomorphism from $(G_1,\delta_1)$ to $(G_2,\delta_2)$.
\end{enumerate}
In particular, $\Upsilon$ is a metric up to metric group isometric isomorphism on the class of proper metric groups.

Moreover, if $(G,\delta_G)$ and $(H,\delta_H)$ are two proper metric monoids with units $e_G$ and $e_H$ then:
\begin{equation*}
\mathrm{GH}((G,\delta_G,e_G),(H,\delta_H,e_H)) \leq \Upsilon((G,\delta_G),(H,\delta_H))\text{,}
\end{equation*}
where $\mathrm{GH}$ is the Gromov-Hausdorff distance for pointed, proper metric spaces \cite{Gromov81}.
\end{theorem}

We now have a metric on the class of proper monoids and a metric on the class of {\qcms s}---the dual propinquity discussed in the previous section. We want to bring them together. We propose to merge the notion of tunnel and the notion of almost isometry as follows.

\begin{definition}[{\cite{Latremoliere18b}}]\label{equi-tunnel-def}
Let $\varepsilon > 0$ and $F$ be a permissible function. Let $(\A_1,\Lip_1,G_1,\allowbreak \delta_1,\alpha_1)$ and $(\A_2,\Lip_2,G_2,\delta_2,\alpha_2)$ be two Lipschitz dynamical $F$-systems. A \emph{$\varepsilon$-covariant $F$-tunnel}
\begin{equation*}
\tau = (\upsilon,\varsigma_1,\varsigma_2)
\end{equation*}
from $(\A_1,\Lip_1,G_1,\delta_1,\alpha_1)$ to $(\A_2,\Lip_2,G_2,\delta_2,\alpha_2)$ is given by an $F$-tunnel $\upsilon$ from $(\A_1,\Lip_1)$ to $(\A_2,\Lip_2)$ and a pair
\begin{equation*}
  (\varsigma_1,\varsigma_2) \in \UIso{\varepsilon}{(G_1,\delta_1)}{(G_2,\delta_2)}{\frac{1}{\varepsilon}} \text{.}
\end{equation*}
\end{definition}

We make two remarks. First, a covariant tunnel does not involve any action on the underlying tunnel: the actions of the domain and codomain are not involved in the definition of the covariant tunnels themselves. We will include these actions in our quantification of  a covariant tunnel later on. Second, we do work with a quantified almost isometry, rather than just a pair of unit-preserving maps. This construct is the path we use to define the covariant propinquity, as it seems to make it easiest to prove such properties as the triangle inequality.

We now quantify covariant tunnels. Of course, covariant tunnels come with a number which is related to the metric $\Upsilon$ by definition, and we also have the extent of the underlying tunnel available to us. What is left is to involve the actual actions in some measurement. The following concept is a generalization of the reach of a tunnel as defined in \cite{Latremoliere13b}. 

\begin{definition}[{\cite{Latremoliere18b}}]\label{reach-def}
  Let $\varepsilon > 0$. Let $\mathds{A}_1 = (\A_1,\Lip_1,G_1,\delta_1,\alpha_1)$ and $\mathds{A}_2 = (\A_2,\Lip_2,\allowbreak G_2,\delta_2,\allowbreak \alpha_2)$ be two Lipschitz dynamical systems. The \emph{$\varepsilon$-reach $\tunnelreach{\tau}{\varepsilon}$} of a $\varepsilon$-covariant tunnel $\tau = (\D,\Lip_\D,\pi_1,\pi_2,\varsigma_1,\varsigma_2)$ from $\mathds{A}_1$ to $\mathds{A}_2$ is given as:
  \begin{equation*}
\max_{\{j,k\}=\{1,2\}}\sup_{\varphi\in\StateSpace(\A_j)} \inf_{\psi\in\StateSpace(\A_k)}\sup_{g \in G_j\left[\frac{1}{\varepsilon}\right]} \Kantorovich{\Lip_\D}(\varphi\circ\alpha_j^g\circ\pi_j, \psi\circ\alpha_k^{\varsigma_j(g)}\circ\pi_k)
\end{equation*}

\end{definition}

We now bring all the data we have so far on covariant tunnels into one quantity.

\begin{definition}[{\cite{Latremoliere18b}}]
The \emph{$\varepsilon$-magnitude} $\tunnelmagnitude{\tau}{\varepsilon}$ of a $\varepsilon$-covariant tunnel $\tau$ is the maximum of its $\varepsilon$-reach and its extent:
  \begin{equation*}
    \tunnelmagnitude{\tau}{\varepsilon} = \max\left\{ \tunnelreach{\tau}{\varepsilon}, \tunnelextent{\tau} \right\} \text{.}
  \end{equation*}
\end{definition}

As with the dual propinquity, we have a natural notion of an appropriate class of covariant tunnels.

\begin{definition}\label{appropriate-def}
  Let $F$ be a permissible function. Let $\mathcal{C}$ be a nonempty class of Lipschitz dynamical $F$-systems. A class $\mathcal{T}$ of covariant $F$-tunnels is \emph{appropriate} for $\mathcal{C}$ when:
  \begin{enumerate}
    \item for all $\mathds{A},\mathds{B} \in \mathcal{C}$, there exists a $\varepsilon$-covariant tunnel from $\mathds{A}$ to $\mathds{B}$ for some $\varepsilon > 0$,
    \item if $\tau \in \mathcal{T}$, then there exist $\mathds{A},\mathds{B} \in \mathcal{C}$ such that $\tau$ is a covariant tunnel from $\mathds{A}$ to $\mathds{B}$,
    \item if $\mathds{A} = (\A,\Lip_\A,G,\delta_G,\alpha)$, $\mathds{B} = (\B,\Lip_\B,H,\delta_H,\beta)$ are elements of $\mathcal{C}$, and if there exists an equivariant full quantum isometry $(\pi,\varsigma)$ from $\mathds{A}$ to $\mathds{B}$, then:
      \begin{equation*}
        \left( \A,\Lip_\A,\mathrm{id}_\A,\pi,\varsigma,\varsigma^{-1} \right), \left( \B,\Lip_\B,\pi^{-1},\mathrm{id}_\B,\varsigma^{-1},\varsigma \right)  \in \mathcal{T}\mathcal{,}
      \end{equation*}
      where $\mathrm{id}_\A$, $\mathrm{id}_\A$ are the identity *-automorphisms of $\A$ and $\B$,
    \item if $\tau = (\D,\Lip,\pi,\rho,\varsigma,\varkappa) \in \mathcal{T}$ then $\tau^{-1} = (\D,\Lip,\rho,\pi,\varkappa,\varsigma) \in \mathcal{T}$,
    \item if $\varepsilon > 0$ and if $\tau_1,\tau_2 \in \mathcal{T}$ are $\frac{\sqrt{2}}{2}$-tunnels, then there exists $\delta \in (0,\varepsilon]$ such that $\tau_1\circ_\delta \tau_2 \in \mathcal{T}$.
  \end{enumerate}
\end{definition}

\begin{notation}
 Let $\mathcal{T}$ be a class of covariant tunnels appropriate for a nonempty class of Lipschitz $F$-dynamical systems $\mathcal{C}$. The set of all $\varepsilon$-covariant $F$-tunnels in $\mathcal{T}$, for any $\varepsilon > 0$, from $\mathds{A}$ to $\mathds{B}$, both chosen in $\mathcal{C}$, will be denoted by:
  \begin{equation*}
    \tunnelset{\mathds{A}}{\mathds{B}}{\mathcal{T}}{\varepsilon} \text{.}
  \end{equation*}
  When $\mathcal{T}$ is simply the class of all Lipschitz $F$-dynamical systems, we shall then write the class of all $F$-tunnels from $\mathds{A}$ to $\mathds{B}$ as:
  \begin{equation*}
        \tunnelset{\mathds{A}}{\mathds{B}}{F}{\varepsilon} \text{.}
  \end{equation*}  
\end{notation}

We now can define the covariant propinquity between Lipschitz dynamical systems.

\begin{definition}[{\cite{Latremoliere18b}}]\label{covariant-propinquity-def}
  Let $\mathcal{C}$ be a nonempty class of Lipschitz $F$-dynamical systems for a permissible function $F$  and let $\mathcal{T}$ be a class of covariant tunnels appropriate for $\mathcal{C}$. For $\mathds{A},\mathds{B} \in \mathcal{C}$, the \emph{covariant $\mathcal{T}$-propinquity} $\covpropinquity{\mathcal{T}}(\mathds{A},\mathds{B})$ is defined as:
\begin{equation*}
  \min\left\{ \frac{\sqrt{2}}{2}, \inf\left\{ \varepsilon > 0 \middle\vert \exists \tau \in \tunnelset{\mathds{A}}{\mathds{B}}{\mathcal{T}}{\varepsilon} \quad \tunnelmagnitude{\tau}{\varepsilon} \leq \varepsilon \right\} \right\} \text{.}
\end{equation*}
\end{definition}

Definition (\ref{covariant-propinquity-def}) indeed defines a metric up to the equivariant full quantum isometries:

\begin{theorem}[{\cite{Latremoliere18b}}]
  Let $\mathcal{C}$ be a nonempty class of Lipschitz $F$-dynamical systems for a permissible function $F$  and let $\mathcal{T}$ be a class of covariant tunnels appropriate for $\mathcal{C}$. If $(\A,\Lip_\A,G,\delta_G,\alpha)$ and $(\B,\Lip_\B,H,\delta_H,\allowbreak \beta)$ in $\mathcal{C}$ then:
   \begin{equation*}
     \covpropinquity{\mathcal{T}}((\A,\Lip_\A,G,\delta_G,\alpha),(\B,\Lip_\B,H,\delta_H,\beta)) = 0
   \end{equation*}
if and only if there exists a full quantum isometry $\pi : (\A,\Lip_\A)\rightarrow(\B,\Lip_\B)$ and an isometric isomorphism of monoids $\varsigma: G\rightarrow H$ such that:
   \begin{equation*}
     \forall g \in G \quad \varphi\circ\alpha^g = \beta^{\varsigma(g)}\circ\varphi \text{,}
   \end{equation*}
i.e. $(\A,\Lip_\A,G,\delta_G,\alpha)$ and $(\B,\Lip_\B,H,\delta_H,\beta)$ are isomorphic as Lipschitz dynamical systems.
\end{theorem}

We now wish to explore the issue of completeness of the covariant propinquity. The important observation in the following theorem is if we start from a convergent sequence of {$F$-\qcms s} for the dual propinquity, and some converging sequence of proper monoids for $\Upsilon$, and if the monoids act on the {\qcms s} entry-wise, then a simple equicontinuity condition, expressed using the notion of dilation for Lipschitz morphisms, is all that is required to get a subsequence of the sequence of Lipschitz dynamical systems thus constructed to converge for the covariant propinquity. In other words, the dual propinquity wants to remember symmetries, as long as they do not degenerate. This idea is captured in its full power in our work in \cite{Latremoliere17c}, which establishes a very general result regarding semigroupoid actions. The following theorem is a consequence of \cite{Latremoliere17c} and captures this idea formally.

\begin{theorem}[{\cite{Latremoliere18c}}]\label{completeness-thm}
  Let $(\A,\Lip)$ be an {$F$-{\qcms}} and let $(G,\delta)$ be a proper monoid. Let $(\A_n,\Lip_n,G_n,\delta_n,\alpha_n)_{n\in\N}$ be a sequence of Lipschitz dynamical systems and let $D : [0,\infty) \rightarrow [0,\infty)$ be a locally bounded function such that:
  \begin{enumerate}
    \item for all $n\in\N$ and $g\in G_n$, we have $\dil{\alpha_n^g} \leq D(\delta_n(e_n,g))$,
    \item $\lim_{n\rightarrow\infty} \Upsilon((G_n,\delta_n),(G,\delta)) = 0$ ,
    \item $\lim_{n\rightarrow\infty} \covpropinquity{}((\A_n,\Lip_n),(\A,\Lip)) = 0$,
    \item for all $\varepsilon>0$, there exists $\omega > 0$ and $N\in\N$ such that if $n\geq N$ and if $g,h \in G_n$ with $\delta_n(g,h) < \omega$, then $\KantorovichDist{\Lip_n}{\alpha_n^g}{\alpha_n^h}{} < \varepsilon$.
  \end{enumerate}
  Then there exists a strictly increasing function $j : \N \rightarrow \N$ and a Lipschitz dynamical system $(\A,\Lip,G,\delta,\alpha)$ such that:
  \begin{equation*}
    \covpropinquity{}((\A_{j(n)},\Lip_{j(n)},G_{j(n)},\delta_{j(n)},\alpha_{j(n)}),(\A,\Lip,G,\delta,\alpha)) \xrightarrow{n\rightarrow\infty} 0 \text{.}
  \end{equation*}

Moreover:
\begin{itemize}
  \item if for all $n\in\N$, the map $\alpha_n$ is a *-endomorphism, then $\alpha$ is also a *-endomorphism,
  \item if for all $n\in\N$, the monoid $G_n$ is a proper group, the map $\alpha_n$ is a full quantum isometry, then $\alpha$ is also a full quantum isometry,
  \item if for all $n\in\N$, the monoid $G_n$ is a compact group, and if the action $\alpha_n$ is ergodic, then $\alpha$ is ergodic as well.
\end{itemize}
\end{theorem}

Before we apply Theorem (\ref{completeness-thm}) to finding sufficient conditions on Cauchy sequences of Lipschitz dynamical systems, we observe some of its more immediate important consequences.

It is usually very difficult to determine the closure, for the dual propinquity, of a given set of {\qcms s}. For instance, note that quantum tori are all members of the closure of all {\Lqcms s} over finite dimensional C*-algebras by \cite{Latremoliere13c}. In that same closure, one also finds all classical compact metric spaces, and noncommutative solenoids. Relaxing the Leibniz inequality to work within some also quasi-Leibniz class of {\qcms s} (specifically, the so-called $(2,1)$-quasi-Leibniz {\qcms s} of \cite{Latremoliere15}), we then find that the closure of finite dimensional {\Qqcms{(2,1)}s} for the dual propinquity contains all AF algebras \cite{Latremoliere15d}, and all {\Lqcms s} whose underlying C*-algebras are nuclear quasi-diagonal by \cite{Latremoliere15}. One technique to compute closures is to invoke a compactness result. For instance, various classes of AF algebras are shown to be compact for the dual propinquity in \cite{Latremoliere15d}. Theorem (\ref{completeness-thm}) offers another technique.

As an example \cite{Latremoliere17c}, let $\mathcal{M}$ be the class of all finite dimensional {\Lqcms s} carrying an ergodic action of $SU(2)$ by quantum isometries. Then by Theorem (\ref{completeness-thm}), any limit of any convergent sequence in $\mathcal{M}$ must also carry an ergodic action of $SU(2)$, so it must be of type I---in fact, it must be a bundle of matrix algebras over a homogeneous space for $SU(2)$. This is a very nontrivial observation showing the power of Theorem (\ref{completeness-thm}).

To study the completeness of the covariant propinquity, it is helpful to first have a general idea of what conditions on Cauchy sequences for $\Upsilon$ make them convergent. The condition we exhibit is a form of equicontinuity for right translations. To formulate this condition, we write:
\begin{equation*}
  \prod_{n\in\N} G_n = \left\{ (g_n)_{n\in\N} : \exists M > 0 \quad \forall n \in \N \quad g_n \in G_n[M] \right\} \text{.}
\end{equation*}

Our equicontinuity condition will be expressed using the following notion of regularity.

\begin{definition}[{\cite{Latremoliere18c}}]\label{regular-sequence-def}
  Let $(G_n,\delta_n)_{n\in\N}$ be a sequence of proper monoids. The set of \emph{regular sequences} $\mathcal{R}((G_n,\delta_n)_{n\in\N})$ is:
      \begin{equation*}
        \left\{ (g_n)_{n\in\N} \in \prod_{n\in\N} G_n \middle\vert 
          \begin{array}{l}
            \forall\varepsilon > 0\quad\exists \omega>0 \quad \exists N \in \N \\
            \forall n \geq N \quad \forall h,k \in G_n \\
            \delta_n(h,k)<\omega \implies \delta_n(h g_n, k g_n) < \varepsilon \text{.}
          \end{array}  \right\} \text{.}
      \end{equation*}
\end{definition}

We can now phrase a sufficient condition for a Cauchy sequence of proper monoids to converge for $\Upsilon$.

\begin{theorem}[{\cite{Latremoliere18c}}]\label{Upsilon-completeness-thm}
  Let $(G_n,\delta_n)_{n\in\N}$ be a sequence such that for all $n\in\N$, there exist $\varepsilon_n > 0$ and
\begin{equation*}
(\varsigma_n,\varkappa_n) \in \UIso{\varepsilon_n}{(G_n,\delta_n)}{(G_{n+1},\delta_{n+1})}{\frac{1}{\varepsilon_n}}
\end{equation*}
such that:
  \begin{enumerate}
    \item $\sum_{n=0}^\infty \varepsilon_n < \infty$,
    \item for all $N\in\N$ and $g \in G_N\left[ \frac{1}{\sum_{n=N}^\infty \varepsilon_n} \right]$:
\begin{equation*}
  \varpi_N(g) = \left(\begin{cases}
      g_n = e_n \text{ if $n < N$,}\\
      g_n = g    \text{ if $n = N$,}\\
      g_n = \varsigma_{n-1}(g_{n-1}) \text{ if $n>N$}
      \end{cases}\right)_{n\in\N} \in \mathcal{R}((G_n,\delta_n)_{n\in\N})\text{.}
  \end{equation*}
  \end{enumerate}
Then there exists a proper monoid $(G,\delta)$ such that $\lim_{n\rightarrow\infty} \Upsilon((G_n,\delta_n),(G,\delta)) = 0$.
\end{theorem}

We note that if we work with proper monoids endowed with bi-invariant metrics, then our regularity condition is automatic since all sequences are regular in the sense of Definition (\ref{regular-sequence-def}). More generally, any reasonable form of uniform control of the Lipschitz constant of right translations can be used to prove that sequences are regular. In \cite{Latremoliere18c}, we show how to exploit this idea  to construct a complete version of $\Upsilon$ on a class of proper monoids including proper monoids with bi-invariant metrics as a proper subclass.

Putting together our compactness theorem and our sufficient condition for convergence of $\Upsilon$-Cauchy sequences of proper monoids, we get the following sufficient condition for convergence of Cauchy sequences in the covariant propinquity.

\begin{corollary}
  Let $F$ be permissible and continuous and let $D : [0,\infty)\rightarrow [0,\infty)$ be a locally bounded function. Let $(\A_n,\Lip_n,G_n,\delta_n,\alpha_n)_{n\in\N}$ be a sequence of Lipschitz dynamical $F$-systems and $(\varepsilon_n)_{n\in\N}$ a sequence of positive real numbers such that for all $n\in\N$, there exists $\varepsilon_n > 0$ and $(\varsigma_n,\varkappa_n) \in \UIso{\varepsilon_n}{G_n}{G_{n+1}}{\frac{1}{\varepsilon_n}}$ and:
  \begin{enumerate}
  \item $\sum_{n=0}^\infty \varepsilon_n < \infty$,
  \item for all $n\in\N$ and $g \in G_n$:
    \begin{equation*}
      \left(\begin{cases}
          g_n = e_n \text{ if $n < N$,}\\
          g_n = g    \text{ if $n = N$,}\\
          g_n = \varsigma_n(g_{n-1}) \text{ if $n>N$}
        \end{cases}\right)_{n\in\N} \in \mathcal{R}((G_n,\delta_n)_{n\in\N})\text{,}
    \end{equation*}
  \item $\forall n\in \N \quad \dpropinquity{}((\A_n,\Lip_n),(\A_{n+1},\Lip_{n+1})) < \varepsilon_n$,
    \item $\Lip_n\circ\alpha_n^{g} \leq D(\delta_n(e_n,g))\Lip_n$,
    \item for all $\varepsilon>0$, there exists $\omega > 0$ and $N\in\N$ such that if $n\geq N$ and if $g,h \in G_n$ with $\delta_n(g,h) < \omega$, then $\KantorovichDist{\Lip_n}{\alpha_n^g}{\alpha_n^h}{} < \varepsilon$.
  \end{enumerate}
Then there exists a Lipschitz dynamical $F$-system $(\A,\Lip_\A,G,\delta,\alpha)$ such that:
\begin{equation*}
  \lim_{n\rightarrow\infty} \covpropinquity{}((\A_n,\Lip_n,G_n,\delta_n,\alpha_n),(\A,\Lip_\A,G,\delta,\alpha)) = 0 \text{.}
\end{equation*}
  Moreover, if for all $n\in\N$, the action $\alpha_n$ is by *-endomorphisms (resp. full quantum isometries, when $G_n$ is a group for all $n\in\N$), then so also is the action $\alpha$.
\end{corollary}

We conclude this section with an example of an explicit covariant convergence. We note that Theorem (\ref{completeness-thm}) provides many examples of convergent subsequences for the covariant propinquity, arising from convergent sequences for the dual propinquity; however these convergent subsequences arise from an implicit construction.

As may be expected, it is helpful to return to the notion of a bridge when working with the covariant propinquity. Let there be given two Lipschitz C*-dynamical systems $(\A,\Lip_\A,G,\delta_G,\alpha)$ and $(\B,\Lip_\B,H,\delta_H,\beta)$, and let us start with a bridge $\gamma = (\D,x,\pi_\A,\pi_\B)$ from $\A$ to $\B$. Let there also be given some pair
\begin{equation*}
  (\varsigma,\varkappa) \in \UIso{\varepsilon}{(G,\delta_G)}{(H,\delta_H)}{\frac{1}{\varepsilon}} \text{.}
\end{equation*}
Now, there is a natural means to define a new seminorm from this data, which incorporates the actions into the bridge seminorm, by setting, for all $(a,b) \in \A\oplus\B$:
\begin{multline*}
  \bridgenorm{\gamma,\varsigma,\varkappa}{a,b} = \max\bigg\{ \bridgenorm{\gamma}{\alpha^g(a),\beta^{\varsigma(g)}(b)}, \bridgenorm{\gamma}{\alpha^{\varkappa(h)}(a),\beta^h(b)} : \\
 g \in G\left[\frac{1}{\varepsilon}\right], h \in H\left[\frac{1}{\varepsilon}\right] \bigg\} \text{.}
\end{multline*}

We can then adjust the notion of the reach of a bridge using our modified bridge seminorm to prove an analogue \cite[Proposition 4.5]{Latremoliere18b} of Theorem (\ref{tunnel-from-bridge-thm}). Using such a technique, we can prove:

\begin{theorem}[{\cite{Latremoliere18b}}]
  Let $\ell$ be a continuous length function on $\T^d$. Let $(G_n)_{n\in\N}$ be a sequence of closed subgroups of $\T^d$ converging to $\T^d$ for the Hausdorff distance induced by $\ell$ on the closed subsets of $\T^d$.

  Let $(\sigma_n)_{n\in\N}$ be a sequence of multipliers of $\Z^d$ converging pointwise to $\sigma$ and such that $\sigma_n$ is the lift of a multiplier of the Pontryagin dual $\widehat{G_n}$ of $G_n$.

  For all $n\in\N$, denote by $\alpha_n$ the dual action of $G_n$ on $C^\ast(\widehat{G_n},\sigma_n)$ and by $\alpha$ the dual action of $\T^d$ on the quantum torus $C^\ast(\Z^d,\sigma)$, and consider the L-seminorms $\Lip_n$ and $\Lip$ induced by Example (\ref{group-lip-ex}).

  Denote by $\covpropinquity{}$ the covariant propinquity for the class of all Leibniz Lipschitz C*-dynamical systems, and identify, for the sake of simplicity, the distance induced by $\ell$ and $\ell$ itself. Then
  \begin{equation*}
    \lim_{n\rightarrow\infty} \covpropinquity{}\left( \left(C^\ast(\widehat{G_n},\sigma_n),\Lip_n,G_n,\ell,\alpha_n\right), \left(C^\ast(\Z^d,\sigma),\Lip,\T^d,\ell,\alpha \right) \right) = 0 \text{.}
  \end{equation*}
\end{theorem}

\bibliographystyle{amsplain}
\bibliography{../thesis}
\vfill

\end{document}